\documentclass[12pt]{article}

\usepackage[margin=1in]{geometry}  
\usepackage{graphicx}              
\usepackage{amsmath}               
\usepackage{amsfonts}              
\usepackage{amsthm,amssymb,amsrefs}        

\newtheorem{thm}{Theorem}[section]
\newtheorem{ex}[thm]{Example}
\newtheorem{prop}[thm]{Proposition}

\newtheorem{conj}[thm]{Conjecture}
\newtheorem{df}[thm]{Definition}
\newcommand{\bnd}{\mathbf{bnd}}
\newcommand{\p}{\mathbf{P}}

\begin{document}

\nocite{*}

\title{Transformations of random walks on groups via Markov stopping time}

\author{Behrang Forghani\footnote{This work is supported by NSERC and  CRC (the Canada research chairs program).}\\
University of Ottawa, ON, Canada}

\maketitle

\begin{abstract}
 We describe a new construction of a family of measures on a group with the same Poisson boundary.
Our approach is based on
 applying  Markov stopping times to an extension of the original random walk.
 \end{abstract}
\section*{Introduction}
In 1963, Furstenberg \cite{Fu63} defined the Poisson boundary of a locally compact group $G$. His definition is based on continuous
bounded harmonic functions on $G$. In these terms, the triviality of the Poisson boundary is equivalent to the absence of nonconstant bounded harmonic functions (the Liouville property).
Later, an equivalent definition of the Poisson boundary appeared in the context
of random walks on groups. The first studies of random walks on groups can be traced to  \cite{DM61}, \cite{G63} and \cite{Ke67}.
However, it was in the early 1980s that great progress in the study of random walks on groups and their Poisson boundaries was made via, for example, the work of Kaimanovich and Vershik \cite{KV83}, Rosenblatt \cite{R81} and Derriennic \cite{D80}.

The simplest definition of the Poisson boundary of a group is due to Kaimanovich (e.g, see \cite{Ka92}).
He defined the Poisson boundary of $(G,\mu)$ as the space of ergodic components of the shift map on the space of sample paths.

The Poisson boundary of $(G,\mu)$ is denoted by $\Gamma(G,\mu)$ and is equipped with a $\mu$-stationary measure $\nu$ (which is often called harmonic), i.e.,
$$
\mu*\nu=\nu.
$$
However, the stationarity condition alone is not enough to  characterize the Poisson boundary. More precisely,
if there is another measure $\mu'$ on the group $G$ such that $\nu$ is $\mu'$-stationary too, then it does not necessarily imply that
$$
\Gamma(G,\mu)=\Gamma(G,\mu').
$$
On the other hand, it is known that the Poisson boundary of $(G,\mu)$ is the same as the Poisson boundary of $(G,\mu^{*n})$ and, more
generally, the Poisson boundary of $(G,\mu')$, where $\mu'$ is an (infinite) convex combination of convolution powers of $\mu$.
These observations give rise to the following problem:

\vspace{.75cm}

 \verb"Problem"~1: \textit{Given a group $G$ and measure $\mu$,
  let $\rho$ be a measure equivalent to the harmonic measure $\nu$ on the Poisson boundary $\Gamma$ of the
 random walk $(G,\mu)$. Does there exist a measure $\mu'$ on $G$ such that the corresponding Poisson boundary $(\Gamma(G,\mu'),\nu')$ is equal to $(\Gamma(G,\mu),\rho)$?}

\vspace{.3cm}

Problem~1 was studied by Furstenberg in \cite{Fu70} and Munchnik in \cite{Mu06}. In this paper, we are interested in the special case
of Problem~1 when $\rho$ is  equal to $\nu$. In other words,

\vspace{.5cm}

\verb"Problem"~2: \textit{Describe all the measures $\mu'$ on the group $G$ such that $\Gamma(G,\mu')=\Gamma(G,\mu)$.}

\vspace{.5cm}

We construct a large class of measures that satisfy the condition of Problem~2. In order to approach this goal, we  replace the original random walk with a new ``extended"  Markov chain and use a Markov stopping time for the extended chain. The set of measures which arises from this method  is closed under two operations: convex combination and convolution.

The paper has three sections. In Section~1, we set up notation and define random walks and
the Poisson boundary.

Section~2  is based on using Markov stopping times. Because of space homogeneity one can iterate a Markov stopping time. Therefore any Markov stopping time of a random walk yields a new random walk such that its Poisson boundary is the same as  the original Poisson boundary. However, the scope of this method is somewhat limited, because, for example, it does not produce random walks determined by
convex combinations of convolution measures (see Example~\ref{convolution}).

In Section~3, we generalize this method by considering
Markov stopping times for an extended Markov chain which ``covers'' the original random walk.
Namely, we extend random walks on a group $G$ to a Markov chain on the product $G\times X$ for an auxiliary  space $X$. This extension does not change the
Poisson boundary, and its advantage is that it gives us more freedom to define  Markov stopping times (Proposition~\ref{extension} and
Theorem~\ref{col}).

An application of this work is to show that the Poisson boundary of
a locally compact group $G$ with a spread-out measure is the same as the Poisson boundary of $G$ with a new measure
which is absolutely continuous with bounded density with respect to the Haar measure on $G$ (Theorem~\ref{willis}).
In a completely different context this fact had also been  proved by Willis \cite{W90} (without boundedness though).

Actually, all theorems in this paper can be proven for locally compact second countable groups, with precisely the same proofs, but for technical simplicity,
we deal just with countable groups.

\section*{Acknowledgment}
I am very grateful to my supervisor, Vadim Kaimanovich,
whose support  and patience have enabled me
to develop an understanding of the subject.

\section{Preliminaries}
In this section we recall the definition of a random walk on a countable group and its Poisson boundary. For more details we refer the reader to \cite{KV83} and \cite{Ka92} as well as the references therein. In this paper $G$ denotes a countable group with a probability measure $\mu$ such that
the support of $\mu$ generates $G$ as a group ( this is not a restrictive assumption, because one can always replace $G$ with the support of $\mu$).
 \begin{df}
 A random walk determined by a measure $\mu$ on a group $G$ is the Markov chain with transition probabilities
 $$
 p(g,h)=\mu(g^{-1}h).
 $$
\end{df}
These transition probabilities give rise to the measure $\p_{\mu}=\p$ on $G^{\infty}$, which is the image of the
Bernoulli measure $\mu^{\infty}$ under the map
$$
(h_1,h_2,\cdots)\mapsto(e,x_1,x_2,\cdots)
$$
where $x_n=h_1h_2\cdots h_n$.
The Lebesgue space $(G^{\infty},\p)$ is called the path (trajectory) space, and $\bar{h}=(h_1,h_2,\cdots)$ is the increment of the path
$(e,x_1,x_2,\cdots)$.
Two paths  $\bar{x}=(e,x_1,x_2,\cdots)$ and $\bar{x'}=(e,x'_1,x'_2,\cdots)$ are equivalent if there exist natural numbers $m$ and $n$ such that
$$
S^m\bar{x}=S^n\bar{x'},
$$
where $S$ is the left shift on the path space. The $\sigma$-algebra corresponding to this equivalence relation is denoted by $\mathcal{A}$ and called the Poisson $\sigma$-algebra.
According to the general Rokhlin theorem (e.g., see \cite{Ro67})
there exist a unique measurable space $\Gamma=\Gamma(G,\mu)$ and measurable map
$\bnd: G^{\infty}\to \Gamma$ such that
$$
 \mathcal{A}=\bnd^{-1}(\sigma(\Gamma)),
$$
where $\sigma(\Gamma)$ is the $\sigma$-algebra on $\Gamma$. Moreover, $\Gamma$ is equipped with the quotient measure
$$
\nu=\mbox{bnd}(\mathbf{P}),
$$
which is called harmonic.
\begin{df}
The space $(\Gamma,\nu)$ is called the Poisson boundary of $(G,\mu)$.
In other words, $(\Gamma,\nu)$ is the space of ergodic components of $S$.
\end{df}
Moreover, $\Gamma$ is endowed with a natural action of the group $G$ in such a way that  the harmonic measure $\nu$ is $\mu$-stationary, i.e.,
$$
\mu*\nu(\gamma)=\sum_g\mu(g)\nu(g^{-1}\gamma)=\nu(\gamma)
$$
for every $\gamma\in\Gamma$.
\begin{df}
A function $f$ on $G$ is $\mu$-harmonic if
$$
f(g)=\sum_hf(gh)\mu(h).
$$
Denote the space of all bounded $\mu$-harmonic functions on $G$  with the supremum  norm by
$H^{\infty}(G,\mu)$.
\end{df}
The following theorem shows the relation between harmonic functions and the Poisson boundary of a random walk.
\begin{thm}[e.g., see \cite{Ka92}]
The map
$$
f(g)= \langle\hat{f},\nu_g\rangle,\ \mbox{where}\ \ \nu_g=g\nu,
$$
establishes an isometric isomorphism of $H^{\infty}(G,\mu)$  and $L^{\infty}(\Gamma,\nu)$.
\end{thm}
The preceding theorem means that the Poisson boundary is trivial if and only if every bounded harmonic function is constant (the Liouville property).
  Dynkin and Maljutov  showed that the Poisson boundary of every nilpotent group for any measure $\mu$ is trivial \cite{DM61}, and  Furstenberg
   proved that every non-amenable group has a non-trivial Poisson boundary \cite{Fu63}. According to Kaimanovich, Vershik and Rosenblatt, for every amenable group there is a symmetric measure such that the corresponding random walk has trivial Poisson boundary (e.g., see \cite{KV83}).

Denote by $\mathcal{P}(G)$  the set of probability measures $\mu$ on $G$ such that the support of  $\mu$ generates $G$ as a group.
Any measure $\mu\in\mathcal{P}(G)$
gives rise to two natural subsets of $\mathcal{P}$(G):
$$
\mathfrak{P}=\mathfrak{P}(G,\mu)=\{\mu'\in \mathcal{P}(G)\ :\ \Gamma(G,\mu)=\Gamma(G,\mu')\}
$$
is the set of all measures whose Poisson boundary coincides (as a measure space) with the Poisson boundary of $\mu$, and
$$
\mathfrak{S}=\mathfrak{S}(G,\mu)=\{\mu'\in \mathcal{P}(G)\ :\ \mu'*\nu=\nu\}
$$ is the set of all measures such that the harmonic measure $\nu$ on $\Gamma(G,\mu)$ is stationary with respect to them.
As mentioned earlier,
$$
\mathfrak{P}\subset\mathfrak{S}.
$$
In some situations the classes $\mathfrak{P}$ and $\mathfrak{S}$ are (nearly) the same. For instance, if $G$ is a hyperbolic group,
then
$$
\mathfrak{P}\cap\mathfrak{M}=\mathfrak{S}\cap\mathfrak{M},
$$
where $\mathfrak{M}$ is the class of measures with a first finite moment \cite{K00}.[The question of the coincidence of the sets $\mathfrak{S}$ and
$\mathfrak{P}$ for  hyperbolic groups in full generality appears to be open, as it is not known whether for measures $\mu$ with infinite entropy
the unique $\mu$-stationary measure on the hyperbolic boundary makes it the Poisson boundary.]

However, this is not the case in general, which follows from the existence of amenable groups with non-trivial Poisson boundaries \cite{KV83}. More precisely,
let $G_k=\Bbb Z^k\rightthreetimes\sum_{\Bbb Z^k}\Bbb Z_2$ be the $k$-dimensional lamplighter group,
 which is amenable. Then, there is a measure $\mu$ such that
$\Gamma(G_k,\mu)$ is trivial and consequently,
$$
\mathfrak{S}(G,\mu)=\mathcal{P}(G).
$$
On the other hand, there are measures on $G_k$ with non-trivial Poisson boundary, so that
$$
\mathfrak{P}(G,\mu)\subsetneqq\mathfrak{S}(G,\mu)
$$

One can easily extend this example to non-amenable groups by considering the
group $H\times G_k$, where $H$ is a non-amenable group. Hence, the question is how big the set $\mathfrak{P}$ can be.

\section{Transformation of a random walk via Markov stopping time}
In this section, we recall the classic definition of a Markov time (see, e.g, \cite{DR84}) and show that how,
 in virtue of space homogeneity, it can be iterated to define a new random walk with the same Poisson boundary as the original random walk.
\begin{df}
A measurable non-negative integer-valued function $T$ on the path space is a Markov time if
$\{x:\ T(\bar{x})=n\}\subset \sigma(x_1,x_2,\cdots,x_n)$ for every $n$. We will require $T$ to be almost surely finite.
\end{df}

A constant function is the simplest example of a Markov time. If $A\subset G$, then the time when the set $A$ is first visited, i.e.,
$$
T_A(\bar{x})=\min_n\{x_n\in A\},
$$
is also a Markov time, which is almost surely finite if and only if the set $A$ is recurrent.

The space homogeneity property of  random walk allows us to define a new random walk by iterating a Markov stopping time.
The transformation of the random walk $(G,\mu)$ via a Markov stopping time can be described by using the Markov stopping time recursively:
\begin{df}
Let $T$ ba a Markov stopping time.
Define $T_1=T$. By induction,
$$
T_{i+1}(\bar{x})=T_{i}(\bar{x})+T(U^{T_i(\bar{x})}(\bar{x})),
$$
where $U$ is the  measure preserving transformation of the path space induced by the left shift in the
space of increments, i.e.,
$$
(U\bar{x})_n=x_1^{-1}x_{n+1}.
$$
\end{df}
Finally, $(x_{T_i})$ is the random walk  governed by the measure $\mu_T$ defined as
$$
\mu_T(h)=\p(x_T=h).
$$
We shall say that the random walk $(G,\mu_T)$ is the transformation of the random walk $(G,\mu)$ via the Markov time $T$.
The aim of the next theorem is to show that a transformation of a random walk via a Markov time does not change the Poisson boundary.
This theorem should be of no surprise to specialists in this area. We apply Doob's optional stopping theorem \cite{DW91}
to prove it.
\begin{thm}\label{Markov}
Let $T$ be a Markov time for the random walk $(G,\mu)$. Then the Poisson boundary of the random walk $(G,\mu)$ is the same as the Poisson boundary
of the random walk $(G,\mu_T)$.
\begin{proof}
It is sufficient to show that
$$
f\in H^{\infty}(G,\mu)\Longleftrightarrow f\in H^{\infty}(G,\mu_T).
$$
\begin{description}
\item[1)] Let $f\in H^{\infty}(G,\mu)$, then $\{f(x_n)\}$ is a martingale sequence. Doob's optional stopping theorem implies that
$$
\mathbb{E}(f(x_1))=\mathbb{E}(f(x_T)).
$$
Hence,
$$
f(x)=\sum_gf(xg)\mu(g)=\sum_gf(xg)\p(x_T=g),
$$
which means that $f$ is a  $\mu_T$-harmonic function.
\item[2)] If $f\in H^{\infty}(G,\mu_T)$, then $f(x)=\int_{\bar{g}}f(\bar{g}_T)\p_x(\bar{g})$. Consequently,
$$
\sum_hf(xh)\mu(h)=\sum_h\int_{\bar{g}}f(\bar{g}_T)d\p_{xh}(\bar{g})\mu(h)=\int_{\bar{g}}f(S(\bar{g})_T)d\p_{x}(\bar{g}),
$$
where $S$ is shift on the path space. By Doob's optional stopping theorem, the last term is equal to $\int_{\bar{g}}f(\bar{g}_T)\p_x(\bar{g})$. Thus, $f$ is $\mu$-harmonic.
\end{description}
 \end{proof}
\end{thm}
The following result is stated without proof in \cite{KV83}.
\begin{ex}\label{convolution}
Let $n$ be a positive integer. Then $\Gamma(G,\mu)=\Gamma(G,\mu^{*n})$
\end{ex}
\begin{proof}
Define  $T$ as the constant function $n$. Then $\mu_T=\mu^{*n}$.
\end{proof}
\begin{ex}
Let $T_1$ and $T_2$ be two Markov times for the random walk $(G,\mu)$. Then $\Gamma(G,\mu)=\Gamma(G,\mu_{T_1}*\mu_{T_2})$.
\end{ex}
\begin{proof}
Let $T=T_1+T_2$. Then $\mu_T=\mu_{T_1}*\mu_{T_2}$.
\end{proof}
\begin{ex}\label{main}
Let $\mu=\alpha+\beta$, where $\alpha$ and $\beta$ are mutually singular and $|\alpha|,\ |\beta|<1$. If $\mu'=(1-\alpha)^{-1}*\beta$, then
the random walks $(G,\mu)$ and $(G,\mu')$ have the same Poisson boundary.
\end{ex}
\begin{proof}
Let $A=\mbox{supp}(\alpha)$ and $B=\mbox{supp}(\beta)$. Define the function $T$  on the path space of the random walk as
$$
T\bar{x}=\min_i\{i>0\ :\ h_i\in B\},
$$
so that $T$ is the first time when the increment of the random walk $(G,\mu)$ belongs to the set $B$.
Hence, Theorem~\ref{Markov} implies that the Poisson boundaries of $(G,\mu)$ and $(G,\mu_T)$ are the same.
By elementary probability considerations and definition of convolution measures then
$\mu_T=(1-\alpha)^{-1}*\beta$.
\end{proof}
Generally speaking, Theorem \ref{Markov} does not produce all the measures from the set of $\mathfrak{P}$. For instance, a convex combination of convolution powers of $\mu$ has been claimed that has the same Poisson boundary as $\mu$ \cite{KV83}. We will also prove this fact later in Section~3.
However, there are convex combinations of convolutions which can not be obtained from Theorem~\ref{Markov}.
\begin{ex}\label{counter}
 Let $G=\Bbb Z_2$ and $\mu=\delta_1$. Then $(0,1,0,1,\cdots)$ is the only path, hence every  Markov stopping time is constant,
 therefore Theorem~\ref{Markov} only produces convolution powers of $\mu$, i.e., $\delta_0$ and $delta_1$. Therefore, the measure
 $\frac{1}{2}(\delta_0+\delta_1)$ can not be obtained in this way.
 \end{ex}

\section{An extension of a random walk}
In this section, we introduce yet another method
 (a generalization of that from Section~2) to construct measures on the group $G$ such that their Poisson boundaries are the same as
 the original Poisson boundary, i.e., they  belong to the set $\mathfrak{P}$ which is introduced in Section~1.

The idea is to extend a random walk on $G$ to a Markov chain on the space  $G\times X$ without changing the Poisson boundary and to use a Markov time on the extension of the random walk. This idea is inspired by Kaimanovich \cite{K92}. We will introduce  new transition probabilities $\pi_{g,x}$ on the space $G\times X$ which are independent of the $x\in X$ and whose projections onto $G$ give the same random walk  on $G$.
\begin{prop}\label{extension}
Let $(X,m)$ be a Lebesgue space. Then the Markov chain with transition probabilities
$$
\pi_{g,x}=\sum_h\mu(h)\delta_{gh}\otimes m
$$
on $(G\times X, \mu\otimes m)$ has the same Poisson boundary as $(G,\mu)$, where $\otimes$ denotes the product measure.
\end{prop}
\begin{proof}
Let $f$ (resp., $F$) be a function on $G$ (resp., $G\times X$). We extend (resp., project) $f$ (resp., $F$) to a  function on $G\times X$ (resp., $G$) as
$$
F(g,x)=f(g).
$$
Since the transition probabilities $\{\pi_{g,x}\}$ do not depend on $x$, a simple calculation shows
that $F$ is a harmonic function with respect to the transition probabilities $\{\pi_{g,x}\}$ if and only if
$f$ is a $\mu$-harmonic function on $G$. Thus, the Poisson boundaries of $G$ and $G\times X$ are the same.
\end{proof}
Theorem~\ref{Markov} was proven for a random walk. However, the proof is
applicable {\it verbatim} to the  Markov chain $\{\pi_{g,x}\}$. Thus, it follows that:
\begin{thm}\label{col}
If $T$ is a Markov time on the space $(G\times X, \mu\otimes m)$, then the projection of the corresponding Markov chain onto $G$ has Poisson boundary $\Gamma(G,\mu)$.
\end{thm}
Note that although the projection of a Markov chain is not necessarily a Markov chain, this is the case
in the preceding theorem.

Example \ref{counter} shows that Theorem~\ref{Markov} does not necessarily produce the convex combinations of convolution powers.
But we can produce them with  the extension method.

\begin{ex}\label{convex}
Let $\mu'=\sum_na_n\mu^{*n}$, where $a_n\geq0$ and $\sum_na_n=1$. Then the
random walks $(G,\mu)$ and $(G,\mu')$ have the same Poisson boundary.
\end{ex}
\begin{proof}
It is enough to define $X$, $m$ and $T$ (as in Theorem~\ref{col}) as $X=\{b_1,b_2,b_3,\cdots\}$, $m=\sum_ia_i\delta_{b_i}$ and
$$
T((h_1,\gamma_1),(h_1h_2,\gamma_2),\cdots)=\gamma_1.
$$
\end{proof}
The second example is the general case of Example~\ref{main}. To prove this, we required the measures $\alpha$
and $\beta$ to be mutually singular. However, with our extension method, this assumption is not necessary.

\begin{ex}\label{main2}
Let $\mu=\alpha+\beta$ be such that $|\alpha|,\ |\beta|<1$. If $\mu'=(1-\alpha)^{-1}*\beta$, then $\Gamma(G,\mu)=\Gamma(G,\mu')$.
\end{ex}
\begin{proof}
We pass from the path
$(x_n)$ to $(x_n,\gamma_n)$, where
$\gamma=(\gamma_n)$ is a sequence of i.i.d. random variables independent of the random walk $(G,\mu)$ and which have the Lebesgue  measure as their common distribution. In other words, $X=(0,1)$ with Lebesgue measure $m$.
It is easy to see that there exists a map $I$ from ${supp}(\mu)$ to the subintervals of $(0,1)$ such that
\begin{itemize}
\item $\{I_{g}\}$ is a partition of the interval $(0,1)$, and
\item for every $g\in G$, there exist intervals $A_g,B_g\subset I_g$ such that
$$
I_g=A_g\cup B_g,\ \ \  A_g\cap B_g=\emptyset,
$$
$\alpha(g)=|A_g|$ and $\beta(g)=|B_g|$.
\end{itemize}
Define a Markov time $T$ as
$$
T=\min_n\{n>0\ :\ \gamma_n\in I(B)\}.
$$
Now, by following the same steps as in Example~\ref{main}, we can show that $\mu'=\mu_T$.
Thus, Theorem~\ref{col} implies that $\Gamma(G,\mu)=\Gamma(G,\mu')$.
\end{proof}
This example has an application to random walks on locally compact groups. In \cite{W90}, Willis introduced the measure $\mu'=(1-\alpha)^{-1}*\beta$, which he used to show that
 $\Gamma(G,\mu)$, where $G$ is a locally compact, second countable group and $\mu$ is a spread-out measure,
 is the same as  $\Gamma(G,\mu'')$, where $\mu''$ is absolutely continuous with respect to the Haar measure on $G$. His proof uses Cohen's factorization and other tools taken from a totally different context. However, using Example~\ref{main2}, we give the same result and can even show that the density function of the absolutely continuous measure with respect to the Haar measure can always be chosen to be bounded.
\begin{thm}\label{willis}
Let $G$ be a locally compact second countable group with Haar measure $m$. If $\mu$ is a spread-out measure on $G$, then there exists a measure $\mu'$ absolutely continuous with respect to $m$ and with bounded density such that $\Gamma(G,\mu)=\Gamma(G,\mu')$.
\end{thm}
\begin{proof}
Since $\mu$ is spread-out, there exists an $n$ such that $\mu^{*n}$ is not singular with respect to $m$. Now,
Lebesgue's decomposition theorem yields
$$
\mu^{*n}=\upsilon+\tau,
$$
where $\tau\prec m$.
 Choose $c>0$ such that the set
$$
B=\{g\in G\ :\ \frac{d\tau}{dm}<c\}
$$
has positive measure. Now, it is enough to choose $\alpha=\upsilon+\tau|_{B^c}$  and $\beta=\tau|_{B}$ as in Example~\ref{main2}.
\end{proof}

We have shown that the set of measures which come  directly  or indirectly (after extension) from a Markov stopping time is closed
under infinite convex combination and convolution. However, we do not know if the set $\mathfrak{P}$ is closed under convolution or convex combination.
Further work in this area could include proving or disproving that all measures in $\mathfrak{P}$ come from measures in Corollary~\ref{col} or
that the set $\mathfrak{P}$ is closed under convex convolution. Hence, there are two conjectures:
\begin{conj}
The set $\mathfrak{P}$ is closed under convolution and convex combination.
\end{conj}
\begin{conj}
Every measure $\mu$ in the set $\mathfrak{P}$ can be obtained using Theorem~\ref{col}.
\end{conj}


\begin{bibdiv}
\begin{biblist}
\bib{D80}{article}{
   author={Derriennic, Yves},
   title={Quelques applications du th\'eor\`eme ergodique sous-additif},
   language={French, with English summary},
   conference={
      title={Conference on Random Walks},
      address={Kleebach},
      date={1979},
   },
   book={
      series={Ast\'erisque},
      volume={74},
      publisher={Soc. Math. France},
      place={Paris},
   },
   date={1980},
   pages={183--201, 4},
   review={\MR{588163 (82e:60013)}},
}

\bib{DM61}{article}{
   author={Dynkin, E. B.},
   author={Maljutov, M. B.},
   title={Random walk on groups with a finite number of generators},
   language={Russian},
   journal={Dokl. Akad. Nauk SSSR},
   volume={137},
   date={1961},
   pages={1042--1045},
   issn={0002-3264},
   review={\MR{0131904 (24 \#A1751)}},
}
\bib{Fu63}{article}{
   author={Furstenberg, Harry},
   title={A Poisson formula for semi-simple Lie groups},
   journal={Ann. of Math. (2)},
   volume={77},
   date={1963},
   pages={335--386},
   issn={0003-486X},
   review={\MR{0146298 (26 \#3820)}},
}
\bib{Fu70}{article}{
   author={Furstenberg, Harry},
   title={Boundaries of Lie groups and discrete subgroups},
   conference={
      title={Actes du Congr\`es International des Math\'ematiciens},
      address={Nice},
      date={1970},
   },
   book={
      publisher={Gauthier-Villars},
      place={Paris},
   },
   date={1971},
   pages={301--306},
   review={\MR{0430160 (55 \#3167)}},
}
\bib{G63}{book}{
   author={Grenander, Ulf},
   title={Probabilities on algebraic structures},
   publisher={John Wiley \& Sons Inc.},
   place={New York},
   date={1963},
   pages={218},
   review={\MR{0206994 (34 \#6810)}},
}
\bib{K92}{article}{
   author={Kaimanovich, Vadim A.},
   title={Discretization of bounded harmonic functions on Riemannian
   manifolds and entropy},
   conference={
      title={Potential theory},
      address={Nagoya},
      date={1990},
   },
   book={
      publisher={de Gruyter},
      place={Berlin},
   },
   date={1992},
   pages={213--223},
   review={\MR{1167237 (94b:31007)}},
}
\bib{Ka92}{article}{
   author={Kaimanovich, Vadim A.},
   title={Measure-theoretic boundaries of Markov chains, $0$-$2$ laws and
   entropy},
   conference={
      title={Harmonic analysis and discrete potential theory},
      address={Frascati},
      date={1991},
   },
   book={
      publisher={Plenum},
      place={New York},
   },
   date={1992},
   pages={145--180},
   review={\MR{1222456 (94h:60099)}},
}
\bib{K00}{article}{
   author={Kaimanovich, Vadim A.},
   title={The Poisson formula for groups with hyperbolic properties},
   journal={Ann. of Math. (2)},
   volume={152},
   date={2000},
   number={3},
   pages={659--692},
   issn={0003-486X},
   review={\MR{1815698 (2002d:60064)}},
   doi={10.2307/2661351},
}
\bib{Ke67}{article}{
   author={Kesten, H.},
   title={The Martin boundary of recurrent random walks on countable groups},
   conference={
      title={Proc. Fifth Berkeley Sympos. Math. Statist. and Probability
      (Berkeley, Calif., 1965/66)},
   },
   book={
      publisher={Univ. California Press},
      place={Berkeley, Calif.},
   },
   date={1967},
   pages={Vol. II: Contributions to Probability Theory, Part 2, pp. 51--74},
   review={\MR{0214137 (35 \#4988)}},
}
\bib{KV83}{article}{
   author={Kaimanovich, V. A.},
   author={Vershik, A. M.},
   title={Random walks on discrete groups: boundary and entropy},
   journal={Ann. Probab.},
   volume={11},
   date={1983},
   number={3},
   pages={457--490},
   issn={0091-1798},
   review={\MR{704539 (85d:60024)}},
}
\bib{Mu06}{article}{
   author={Muchnik, Roman},
   title={A note on stationarity of spherical measures},
   journal={Israel J. Math.},
   volume={152},
   date={2006},
   pages={271--283},
   issn={0021-2172},
   review={\MR{2214464 (2006m:60110)}},
   doi={10.1007/BF02771987},
}
\bib{DR84}{book}{
   author={Revuz, D.},
   title={Markov chains},
   series={North-Holland Mathematical Library},
   volume={11},
   edition={2},
   publisher={North-Holland Publishing Co.},
   place={Amsterdam},
   date={1984},
   pages={xi+374},
   isbn={0-444-86400-8},
   review={\MR{758799 (86a:60097)}},
}
\bib{Ro67}{article}{
   author={Rohlin, V. A.},
   title={Lectures on the entropy theory of transformations with invariant
   measure},
   language={Russian},
   journal={Uspehi Mat. Nauk},
   volume={22},
   date={1967},
   number={5 (137)},
   pages={3--56},
   issn={0042-1316},
   review={\MR{0217258 (36 \#349)}},
}
\bib{R81}{article}{
   author={Rosenblatt, Joseph},
   title={Ergodic and mixing random walks on locally compact groups},
   journal={Math. Ann.},
   volume={257},
   date={1981},
   number={1},
   pages={31--42},
   issn={0025-5831},
   review={\MR{630645 (83f:43002)}},
   doi={10.1007/BF01450653},
}
\bib{W90}{article}{
   author={Willis, G. A.},
   title={Probability measures on groups and some related ideals in group
   algebras},
   journal={J. Funct. Anal.},
   volume={92},
   date={1990},
   number={1},
   pages={202--263},
   issn={0022-1236},
   review={\MR{1064694 (91i:43003)}},
   doi={10.1016/0022-1236(90)90075-V},
}
\bib{DW91}{book}{
   author={Williams, David},
   title={Probability with martingales},
   series={Cambridge Mathematical Textbooks},
   publisher={Cambridge University Press},
   place={Cambridge},
   date={1991},
   pages={xvi+251},
   isbn={0-521-40455-X},
   isbn={0-521-40605-6},
   review={\MR{1155402 (93d:60002)}},
}
\end{biblist}
\end{bibdiv}
\end{document}